\documentclass[12pt]{article}

\usepackage{authblk}
\usepackage{comment}
\usepackage{amssymb}
\usepackage{amsthm}
\usepackage{amscd}
\usepackage{graphicx}
\usepackage{titlesec}
\titleformat{\section}{\bfseries}{\thesection .}{0.5em}{}
\titleformat{\subsection}{\itshape}{\thesubsection .}{0.5em}{}
\usepackage[top=25.4mm, bottom=25.4mm, left=31.7mm, right=32.2mm]{geometry}
\usepackage[tbtags]{amsmath}
\usepackage{booktabs}
\allowdisplaybreaks  
\usepackage[numbers,sort&compress]{natbib}

\usepackage{caption}
\usepackage{amsmath}
\usepackage{amsfonts}
\usepackage{graphicx,amsmath,amsfonts,amssymb,amscd,amsthm}
\usepackage{enumerate}
\newtheoremstyle{kai}
{3pt} {3pt} {} {} {\bfseries} {.} {.5em} {}
\def\EquationsBySection{\def\theequation
{\thesection.\arabic{equation}}%
\@addtoreset{equation}{section}}

\newcommand\old[1]{}
\newcommand{\pend}{\hfill \thicklines \framebox(6.6,6.6)[l]{}}
\renewenvironment{proof}{\noindent {\it  Proof.} \rm}{\pend}
\newtheorem{theorem}{Theorem}[section]
\newtheorem{lemma}{Lemma}[section]
\newtheorem{corollary}{Corollary}[section]
\newtheorem{proposition}{Proposition}[section]
\newtheorem{remark}{Remark}[section]
\newtheorem{definition}{Definition}[section]

\EquationsBySection \makeatother
\usepackage{indentfirst}

\begin{document}
\pagestyle{plain}
\title
{\bf L\'{e}vy-driven Fluid Queue with Server Breakdowns and
Vacations}
\author{Jinbiao Wu $^a$\thanks{E-mail address: wujinbiao@ymail.com}, {Zaiming Liu$^a$}, {Yi Peng$^b$}
\\
\vspace{0.2cm}\small\it $^a$School of Mathematics and Statistics, \\
 \small\it Central South University,
Changsha 410083, Hunan, P.R.  China\\
 \small\it  $^b$Junior Education Department,\\ \small\it Changsha Normal University, Changsha
410100,  Hunan, P.R.  China}

\date{}
\maketitle

\begin{abstract}
In this paper, we consider a L\'{e}vy-driven fluid queueing system
where the server may subject to breakdowns and repairs. In addition,
the server will leave for a vacation each time when he finds an
empty system. We cast the queueing process as a L\'{e}vy process
modified to have random jumps at two classes of stopping times. By
using the Kella-Whitt martingale method, we obtain the limiting
distribution of the virtual waiting time process. Moreover, we
investigate the busy period, the correlation structure and the
stochastic decomposition properties. These results may be
generalized to L\'{e}vy processes with multi-class jump inputs or
L\'{e}vy-driven queues with multiple input classes.
\end{abstract}



\section{Introduction}\label{sec1}
The class of L\'{e}vy processes consisting of all stochastic
processes with stationary independent increments is one of the most
important family of stochastic processes arising in many areas of
applied probability. It covers well-studied processes such as
Brownian motion and compound Poisson processes. L\'{e}vy processes
are used as models in the study of queueing systems, insurance
risks, storage systems, mathematical finance and so on. For a
comprehensive and outstanding analysis of L\'{e}vy processes and
their applications, readers can refer to the books by Bertoin
\cite{Bertoin1998}, Sato \cite{Sato1999}, Applebaum
\cite{Applebaum2004}, and Kyprianou \cite{Kyprianou2006}.

Over recent years it has been a rapid growth in the literature on
queues with server breakdowns and vacations due to their widely
applications in computer communication networks and manufacturing
systems. Li et al. \cite{Li1997} considered an $M/G/1$ queue with
Bernoulli vacations and server breakdowns using a supplementary
variable method. Gray et al. \cite{Gray2000} analyzed a
multiple-vacation $M/M/1$ queueing model, where the service station
is subject to breakdown while in operation. Ke \cite{Ke2003} studied
the control policy in $M^{[X]}/M/1$ queue with server breakdowns and
multiple vacations. Jain and Jain \cite{Jain2010} dealt with a
single server working vacation queueing model with multiple types of
server breakdowns.  Wu and Yin \cite{Wu2011} gave a detailed
analysis on an $M/G/1$ retrial queue with non-exhaustive random
vacations and unreliable server. Wang et al. \cite{Wang2013} studied
a discrete-time $Geo/G/1$ queue in which the server operates
multiple vacations and may break down while working. Yang and Wu
\cite{Yang2015} investigated the N-policy $M/M/1$ queueing system
with working vacation and server breakdowns. However, most of the
papers are focused on queues where customers arrive at the system
according to independent Poisson processes or geometrical arrival
processes.  Queueing systems with L\'{e}vy input (or L\'{e}vy-driven
queues) are still not well investigated in literature. Recently,
queuing systems with L\'{e}vy input have been attracting increasing
attention in the applied probability and stochastic operations
research communities. L\'{e}vy-driven queues covers the classical
$M/G/1$ queue and the reflected Brownian motion as special cases.
Kella and Boxma \cite{Kella1991} first considered L\'{e}vy processes
with secondary jump input which were applied to analyze queues with
server vacations. Recently, Lieshout and Mandjes \cite{Lieshout2008}
analyzed tail asymptotics of a two-node tandem queue with spectrally
positive L\'{e}vy input. Boxma et al. \cite{Boxma2010} analyzed a
generic class of L\'{e}vy-driven queuing systems with server
vacation. D\c{e}bicki and Mandjes \cite{Debicki2012} provided a
survey on L\'{e}vy-driven queues. D\c{e}bicki et al.
\cite{Debicki2013} focused on transient analysis of L\'{e}vy-driven
tandem queues. Palmowski et al. \cite{Palmowski2014} considered a
controlled fluid queuing model with L\'{e}vy input. Boxma and Kella
\cite{Boxma2014} generalized known workload decomposition results
for L\'{e}vy queues with secondary jump inputs and queues with
server vacations or service interruptions. However, to the best of
our knowledge,  no work on L\'{e}vy queues with server breakdowns
and vacations is found in the queueing literature.

In this paper, we consider a single-server L\'{e}vy-driven fluid
queue with multiple vacations (exhaustive service) and a server
subject to breakdowns and repairs which is motivated by the
performance analysis of resources in communication networks. A
breakdown at the server is represented by the arrival of a failure.
The principal purpose of the present paper is to apply the
martingale results which were derived in Kella and Whitt
\cite{Kella1992} to investigate the stochastic dynamics of the
system and realize an extensive analysis of the system from the
transient virtual waiting time process to the steady-state
distribution of the waiting time process. In addition, we give a
detailed analysis of the system busy period and the queue's
correlation structure. Furthermore, we establish decomposition
results for the model.

The rest of the paper is organized as follows. In Section 2, we
introduce a few basic facts concerning L\'{e}vy processes and
martingales. In Section 3, we describe the general model and
formulate the model as a L\'{e}vy process modified to have random
jumps at two classes of stopping times. In Section 4, we
characterize the steady-state distribution of the virtual waiting
time process and the mean length of the busy period. In Section 5,
we address the transient distribution of the waiting time process as
well as the queue's correlation structure. In Section 6, we present
two stochastic decomposition results.

 \section{Preliminaries on L\'{e}vy processes}\label{sec2}

In this section, we mainly give several notations, definitions and
propositions about L\'{e}vy processes. For more details, see
\cite{Bertoin1998, Sato1999, Applebaum2004, Asmussen2003}.

Denote $\mathbb{R}=(-\infty, \infty)$ and $\mathbb{R}_+=[0,
\infty)$. Let $(\Omega, \mathcal{F}, \mathbb{P}, \{
\mathcal{F}_{t\geq0}\})$ be a complete probability space endowed
with a $\sigma$-field filtration $\mathcal{F}_{t\geq0}$, i.e. an
increasing family of sub-fields, which fulfils the usual conditions.
That is, each $\mathcal{F}_t$ is $\mathbb{P}$-complete and
$\mathcal{F}_t=\cap_{s>t}\mathcal{F}_s$ for every $t$. Throughout,
adapted, stopping times and martingales will be defined with respect
to this filtration.
\begin{definition}
Let $X=(X_t, t\geq0)$ be a stochastic process defined on a
probability space $(\Omega, \mathcal{F}, \mathbb{P})$, if
\begin{enumerate}[(i)]
\item $X_0=0$ (a.s.);
\item $X$ has independent and stationary increments;
\item $X$ is stochastically continuous, i.e. for all $a>0$ and for
all $s\geq0$,
$$\lim_{t\to s}\mathbb{P}(|X_t-X_s|>a)=0,$$
then $X$ is a L\'{e}vy process.
\end{enumerate}
\end{definition}

Two central and key results in the foundations of L\'{e}vy processes
are given by the following propositions (known as the
L\'{e}vy-Khintchine formula and L\'{e}vy-It\^{o} decomposition
respectively).
\begin{proposition}
Let $X$ be a L\'{e}vy process with L\'{e}vy measure $\nu$
($\int_{\mathbb{R}-\{0\}} (|x|^2\wedge1)\nu(dx)<\infty$). Then
$$\mathbb{E}e^{i\theta X_t}=e^{t\psi(\theta)}, \ \ \ t\geq0,$$
where
$$\psi(\theta)=i\mu\theta-\frac{1}{2}\theta^2\sigma^2+\int_{\mathbb{R}-\{0\}}(e^{i\theta x}-1
-i\theta x\mathbf{1}_{|x|<1})\nu(dx),$$ $\mu\in \mathbb{R}$ and
$\sigma\geq0$. Moreover, given $\nu$, $\mu$, $\sigma^2$, the
corresponding L\'{e}vy process is unique in distribution.
Furthermore, the jump process of $X$ is a Poisson point process with
characteristic measure $\nu$, that is
 $$\nu(\Lambda)=\mathbb{E}\left[\sum_{0<s\leq1}\mathbf{1}_\Lambda(\Delta X_s)\right].$$
 \end{proposition}
The function $\psi(\cdot)$ is called the L\'{e}vy exponent of the
L\'{e}vy process.
\begin{proposition}
Let $X$ be a L\'{e}vy process. Then $X$ has a decomposition
\begin{align*}
X_t=B_t+\mu t+\int_{\{|x|<1\}}x(N_t(\cdot,dx)-t\nu(dx))
+\int_{\{|x|\geq1\}}xN_t(\cdot,dx)
\end{align*}
where $B$ is a Brownian motion and for $0\notin \bar{\Lambda}$,
$N_t^\Lambda=\int_\Lambda N_t(\cdot,dx)$ is a Poisson process
independent of $B$.
 \end{proposition}

The following are some known fundamental results about L\'{e}vy
processes:
\begin{enumerate}[(i)]
\item Any L\'{e}vy process can be represented as an independent sum of a Brownian motion $B$ and a
 `compound Poisson'-like process $X^0$. If $X$ has no negative jumps and the paths of $X^0$
are of bounded variation, then $X^0$ can be a subordinator which can
be represented as a nonnegative compound Poisson process or as the
limit of a sequence of nonnegative compound Poisson processes
\cite{Asmussen2010}.
\item If $\nu((-\infty,0))=0$, then the Laplace-Stieltjes
transform exists and is given by $\mathbb{E}\exp(-\theta
X_t)=\exp(t\varphi(\theta))$ where
\begin{align*}
\varphi(\theta)=\log \mathbb{E}\exp\{-\theta X_1\}
=-\mu\theta+\frac{1}{2}\theta^2\sigma^2+\int_{\mathbb{R}_+}(e^{-\theta
x}-1 +\theta x\mathbf{1}_{|x|<1})\nu(dx),
\end{align*}
here the function $\varphi(\cdot)$ is called the Laplace-Stieltjes
exponent of the L\'{e}vy process. It is easy get that $\varphi(0)=0$
and $\varphi(\cdot)$ is convex; If $X$ is not a subordinator, then
$\varphi(\theta)\to\infty$ as $\theta\to\infty$; Whenever $EX_1<0$,
$\varphi(\cdot)$ is strictly increasing. Throughout the paper,
exponent will mean the Laplace-Stieltjes exponent $\varphi(\cdot)$
due to the convenience compared to the L\'{e}vy exponent
$\psi(\cdot)$.
\item If $X$ has bounded jumps, then $\mathbb{E}|X_1|^n<\infty$ for every $n\geq0$.
\end{enumerate}

It is known that there always exists a version with sample paths  of
$X$ that are c\`{a}dl\`{a}g, which is therefore strong Markov.
Throughout this paper, every L\'{e}vy process mentioned is assumed
to be such.

In this paper, we focus only on the spectrally positive L\'{e}vy
process, i.e. $\nu(-\infty,0)=0$. For any spectrally positive
L\'{e}vy process $X$ with $\beta>0$ and $EX_1<0$. Denote by
$T^\xi=\inf\{t|X_t=-\xi\}$ for any non-negative random variable
$\xi$ which is independent of $X$. Then it follows Theorem 3.12 in
\cite{Kyprianou2006} that
$$E\exp\{-\beta T^\xi\}=\exp\{-\varphi^{-1}(\beta)\xi\}, \ \ \
ET^\xi=\frac{E\xi}{-EX_1}.$$

We now give the generalized Pollaczek-Khinchine formula associated
with a reflected L\'{e}vy process with no negative jumps which is
also a celebrated formula in queueing theory. Given some random
variable $0\leq R_0\in \mathcal{F}_0$, let
$$I_t=\left(-\inf_{0\leq s\leq t}X_s-R_0\right)^+, \ \ \
R_t=R_0+X_t+I_t, \ \ \ t\geq0.$$ Then $R$ is a reflected L\'{e}vy
process with $I$ being its local time at zero. Since $X$ has no
negative jumps, $I$ is continuous with $I_0=0$ and is the minimal
right-continuous nondecreasing process such that $R_t\geq0$ for all
$t$. The process $\{I_t, \ t\geq0\}$ must only increase when
$R_t=0$, so that
$$\int_0^\infty \mathbf{1}_{\{R_t>0\}}d I_t=0.$$

The following result is the famous (generalized) Pollaczek-Khinchine
formula in queueing theory. We can refer to the papers
\cite{Kella1991} and \cite{Debicki2012} for its proof.
\begin{proposition}
If $X$ is a spectrally positive L\'{e}vy process such that $EX_1<0$,
then
$$\lim_{t\to\infty}\mathbb{E}e^{-\theta R_t}=\frac{\theta \varphi'(0)}{\varphi(\theta)}, \ \ \ \theta>0.$$
\end{proposition}

 We end this section by presenting
the famous Kella-Whitt martingale associated with spectrally
positive L\'{e}vy process $X$ with exponent $\varphi(\theta)$ that
we will apply it to analyze our fluid queue.
\begin{proposition}\label{pro2.4}
Let $X$ be a spectrally positive L\'{e}vy process with exponent
$\varphi(\theta)$. Let $L_t=\int_0^t dL_s^c+\sum_{0\leq s\leq
t}\Delta L_s$ be an adapted c\`{a}dl\`{a}g process of bounded
variation on finite intervals with continuous part $\{L_t^c\}$ and
jumps $\Delta L_s=L_s-L_{s-}$, and define $Z_t=Z_0+X_t+L_t$, where
$Z_0\in\mathcal{F}_0$. Then
\begin{align*}
M_t=&\varphi(\theta)\int_0^t \exp\{-\theta Z_s\}ds+\exp\{-\theta
Z_0\}-\exp\{-\theta Z_t\} -\theta \int_0^t\exp\{-\theta Z_s\}
dY_s^c\\&+\sum_{0\leq s\leq t}[\exp\{-\theta Z_s\}-\exp\{-\theta
Z_s-\Delta L_s\}]
\end{align*}
is a local martingale. In addition, if the expected variation of
$\{L_t^c, t\geq0\}$ and the expected number of jumps of $\{L_t,
t\geq0\}$ are finite on every finite interval and $\{Z_t, t\geq0\}$
is a non-negative process (or bounded below), then $\{M_t, t\geq0\}$
is a martingale.
\end{proposition}
The Kella-Whitt martingale is a stochastic integral with respect to
the Wald martingale $N_t=\exp\{-\theta X_t-\varphi(\theta)t\}$ and
the readers may refer to the paper \cite{Kella1992} for its proof.

\section{Model formulation}\label{sec3}
We consider a single fluid queue with a (nondecreasing) L\'{e}vy
input and available-processing (or service) processes. For any
$t\geq0$, let $X_t$ be the cumulative input of fluid over the
interval $[0,t]$. We assume that $X$ is a subordinator, that is, a
L\'{e}vy process, having the exponent
$$\phi(\theta)=-a\theta+\int_{\mathbb{R}_+}(e^{-\theta
x}-1)\nu(dx),$$ where $a\geq0$ and
$\int_{\mathbb{R}_+}x\nu(dx)<\infty$.
 Let $S(t)$ be the cumulative
available processing over the interval $[0,t]$. Here, we assume the
server processes the fluid at a constant deterministic rate $r$
(whenever the fluid level is positive). So that $S(t)=rt, \ \
t\geq0$. We assume the storage space is unlimited. Let $\{\tau_n,
n\geq1\}$ be the strictly increasing sequence of stopping times at
which the server fails. When the server fails, it is repaired
immediately and the time required to repair it is a positive random
variable $\xi_n $ which is $\mathcal{F}_{\tau_n}$-measurable for
$n\geq1$. When the system becomes empty, the server takes a vacation
of random length $\eta_n$ which is
$\mathcal{F}_{\sigma_n}$-measurable for $n\geq1$, where $\{\sigma_n,
n\geq1\}$ is a strictly increasing sequence of stopping times at
which the server is leaving for a vacation. Vacations continue
until, on return from a vacation, the server finds the system is
non-empty.

Let $W_t$ represent the virtual waiting time at time $t$. Denote the
initial workload by $W_0\in\mathcal{F}_0$. Then $W_t$ can be defined
by (the empty sum is zero)
\begin{align}\label{eq3.1}
W_t=W_0+X_t-rt+\sum_{i=1}^{N_t^R}\xi_i+\sum_{i=1}^{N_t^V}\eta_i, \ \
\ t\geq0,\end{align}
 where
$$N_t^R=\sup\{n|\tau_n\leq t\}, \ \ \ N_t^V=\sup\{n|\sigma_n\leq
t\},$$
$$\sigma_n=\inf\left\{t\geq0\bigg|W_0+X_t-rt+\sum_{i=1}^{N_t^R}\xi_i+\sum_{i=1}^{n-1}\eta_i=0\right\}, \ \ \ n\geq1.$$

\begin{remark}
It should be noted that here $W_t$, $t\geq0$, is not the workload
process. However, if $\xi_n$ represents the total arriving workload
during the $n$th repair period and $\eta_n$ corresponds to the total
arriving workload during the $n$th vacation period, then $W_t$ is
the workload in the system conditioned on being in the active
periods which are defined as the normal processing fluid periods
(excluding the repair periods and vacation periods).
\end{remark}

Define $Y_t=X_t-rt$. Then $Y$ is a spectrally positive L\'{e}vy
process with exponent $\varphi(\theta)=\phi(\theta)+r\theta$. Denote
$\rho=-\phi'(0)$ and impose the conditions $\rho<r$ and
$\mathbb{E}|Y_t|<\infty$ for all $t\geq0$.

\section{The steady-state distribution}\label{sec4}
In this section, we characterize the limiting distribution of the
virtual waiting time $W_t$. Let $\overset{d}{\to}$ denote
convergence in distribution and let $\overset{p}{\to}$ denote
convergence in probability. Under the condition $\rho<r$, we suppose
$W_t\overset{d}{\to}W$ and $t^{-1}\mathbb{E}W_t\to0$, as
$t\to\infty$, where $W$ is a random variable.

In order to derive the steady-state distribution of the virtual
waiting time $W_t$, we give the following lemmas.
\begin{lemma}\label{lem4.1}
If $E(N_t^R+N_t^V)<\infty$ for all $t$, then $\{M_t, t\geq0\}$ is a
zero-mean real-valued martingale with respect to $\{\mathcal{F}_t,
t\geq0\}$, where
\begin{align}\label{eq4.1}
M_t=\varphi(\theta)\int_0^t e^{-\theta W_s}ds+e^{-\theta
W_0}-e^{-\theta W_t}
-\sum_{k=1}^{N_t^R}\left[e^{-\theta(W_{\tau_k}-\xi_k)}-e^{-\theta
W_{\tau_k}}\right]-\sum_{k=1}^{N_t^V}(1-e^{-\theta\eta_k}).
\end{align}
\end{lemma}
\begin{proof}
Let $H_t=\sum_{i=1}^{N_t^R}\xi_i+\sum_{i=1}^{N_t^V}\eta_i$. Then
$W_t=W_0+Y_t+H_t$ and $\{H_t, t\geq\}$ is adapted of bounded
variation on every finite interval, and $H_t^c=0$ for all $t\geq0$.
Since $E(N_t^R+N_t^V)<\infty$, by Proposition \ref{pro2.4} and
through a trivial manipulation, we get that \eqref{eq4.1} is a
martingale.
\end{proof}

\begin{lemma}\label{lem4.2}
Let $T$ is a nonnegative random variable with possible integer
values $0,1,2, \cdots$ and $\mathbb{E}T<\infty$. If $\{\zeta_n,
n\geq1\}$ is an i.i.d. sequence with $\zeta_n$ independent of
$\mathbf{1}_{\{T\geq n\}}$ and $\mathbb{E}\zeta_1<\infty$, then
$$\mathbb{E}\sum_{n=1}^T\zeta_n=(\mathbb{E}\zeta_1)(\mathbb{E}T).$$
\end{lemma}
\begin{proof}
Since $\zeta_n$ is independent of $\mathbf{1}_{\{T\geq n\}}$, the
bounded convergence theorem yields
\begin{align*}
\mathbb{E}[\sum_{n=1}^T\zeta_n]&=\mathbb{E}[\sum_{n=1}^\infty\zeta_n\mathbf{1}_{\{T\geq
n\}}]=\sum_{n=1}^\infty(\mathbb{E}\zeta_n)(\mathbb{E}\mathbf{1}_{\{T\geq
n\}})=\mathbb{E}\zeta_1\sum_{n=1}^\infty(\mathbb{E}\mathbf{1}_{\{T\geq n\}})\\
&=\mathbb{E}\zeta_1\sum_{n=1}^\infty \mathbb{P}(T\geq
n)=(\mathbb{E}\zeta_1)(\mathbb{E}T).
\end{align*}
\end{proof}

With the help of the above two lemmas, we can derive the limiting
distribution of the virtual waiting time provided in the following
theorem which is the most important result of this paper.
\begin{theorem}\label{thm4.1}
Let $\{\xi_n, n\geq1\}$ and $\{\eta_n, n\geq1\}$ be two positive
i.i.d. sequences with  $\mathbb{E}\xi_1<\infty$ and
$\mathbb{E}\eta_1<\infty$, respectively. Suppose that
$n^{-1}\tau_n\overset{p}{\to}\lambda_R^{-1}$ and
$n^{-1}\sigma_n\overset{p}{\to}\lambda_V^{-1}$ as $n\to\infty$ for
$0<\lambda_R^{-1}<\infty$ and $0<\lambda_V^{-1}<\infty$,
respectively. Further, assume that $\{W_{\tau_k}, k\geq1\}$ and
$\{W_{\tau_k}-\xi_k, k\geq1\}$ are stationary and ergodic that
$(W_{\tau_k}-\xi_k, W_{\tau_k})\overset{d}{\to}(W^-,W^+)$ where
$W^-,W^+$ are two proper random variables. If $\xi_n$ is independent
of $\mathbf{1}_{\{N_1^R\geq n\}}$ and $\eta_n$ is independent of
$\mathbf{1}_{\{N_1^V\geq n\}}$, then for $\theta>0$,
\begin{align}\label{eq4.2}
\lambda_R=\frac{p\varphi'(0)}{\mathbb{E}\xi_1}, \ \ \
\lambda_V=\frac{(1-p)\varphi'(0)}{\mathbb{E}\eta_1},
\end{align}
\begin{align}\label{eq4.3}
\lim_{t\to\infty}t^{-1}\mathbb{E}\sum_{k=1}^{N_t^R}\xi_k=-p\mathbb{E}Y_1,
\ \ \
\lim_{t\to\infty}t^{-1}\mathbb{E}\sum_{k=1}^{N_t^R}\eta_k=-(1-p)\mathbb{E}Y_1,
\end{align}
\begin{align}\label{eq4.4}
\lim_{t\to\infty}\mathbb{E}e^{-\theta W_t}=\mathbb{E}e^{-\theta
W}=\frac{\theta\varphi'(0)}{\varphi(\theta)}\left[p\frac{\mathbb{E}e^{-\theta
W^-}-\mathbb{E}e^{-\theta
W^+}}{\theta\mathbb{E}\xi_1}+(1-p)\frac{1-\mathbb{E}e^{-\theta\eta_1}}{\theta\mathbb{E}\eta_1}\right],
\end{align}
where $0\leq p\leq1$.
\end{theorem}
\begin{proof}
Dividing \eqref{eq3.1} by $t$ and taking expectation, we get
\begin{align}\label{eq4.5}
t^{-1}\mathbb{E}W_t=t^{-1}\mathbb{E}W_1+\mathbb{E}Y_1+t^{-1}\mathbb{E}\sum_{k=1}^{N_t^R}\xi_k+t^{-1}\mathbb{E}\sum_{k=1}^{N_t^V}\eta_k.
\end{align}
where $t^{-1}\mathbb{E}W_t\to0$ by the assumption. Thus, we have
proved \eqref{eq4.3}.

 Taking $W_0=W^*$ where $W^*$ is an independent copy of $W$, $\{W_t,
 t\geq0\}$ becomes stationary. By Lemma \ref{lem4.2} and \eqref{eq3.1}, we get
 \begin{align}\label{eq4.6}
\mathbb{E}N_1^R\mathbb{E}\xi_1+\mathbb{E}N_1^V\mathbb{E}\eta_1=-\mathbb{E}Y_1.
 \end{align}
So that,
\begin{align*}
\mathbb{E}N_1^R\mathbb{E}\xi_1=-p\mathbb{E}Y_1, \ \ \
\mathbb{E}N_1^V\mathbb{E}\eta_1=-(1-p)\mathbb{E}Y_1.
\end{align*}
Hence,
\begin{align}\label{eq4.7}
\mathbb{E}N_1^R=\frac{p\varphi'(0)}{\mathbb{E}\xi_1}, \ \ \
\mathbb{E}N_1^V=\frac{(1-p)\varphi'(0)}{\mathbb{E}\eta_1}.
\end{align}

Conditions $n^{-1}\tau_n\overset{p}{\to}\lambda_R^{-1}$ and
$n^{-1}\sigma_n\overset{p}{\to}\lambda_V^{-1}$ as $n\to\infty$ imply
that $t^{-1}N_t^R\overset{p}{\to}\lambda_R$ and
$t^{-1}N_t^V\overset{p}{\to}\lambda_V$. Together with
$t^{-1}N_t^R\to \mathbb{E}N_1^R$ and $t^{-1}N_t^V\to
\mathbb{E}N_1^V$ w.p.1 as $t\to\infty$, this implies \eqref{eq4.2}.

 From Lemma \ref{lem4.2}, optional stopping at $t=1$ in \eqref{eq4.1}
yields
\begin{align}\label{eq4.8}
0=\varphi(\theta)\mathbb{E}e^{-\theta
W}-\mathbb{E}N_1^R\left[\mathbb{E}e^{-\theta
W^-}-\mathbb{E}e^{-\theta
W^+}\right]-\mathbb{E}N_1^V(1-\mathbb{E}e^{-\theta\eta_1})
\end{align}
Substituting \eqref{eq4.7} into \eqref{eq4.8}, we obtain
\eqref{eq4.4}.

\end{proof}

\begin{remark}
From \cite{Chung1974}, we have $\varphi(\theta)=0$ for some
$\theta>0$ if and only if $Y_t$ has a lattice distribution. But for
a  L\'{e}vy process with no negative jumps and $\mathbb{E}Y_t<0$, it
is not possible. In addition, when $N_t^R\equiv0$ for all $t$, i.e.,
$p=0$, the present system reduces to the L\'{e}vy process-driven
queue with server vacations. It may be noted that the equation
\eqref{eq4.4} after putting $p=0$ agrees with the equation (4.6)
presented in Kella and Whitt \cite{Kella1991}. We should note that
when $N_t^V\equiv0$ for all $t$, the present system reduces to the
L\'{e}vy process-driven queue with server breakdowns. However, the
equation \eqref{eq4.4} in Theorem \ref{thm4.1} after putting $p=1$
cannot characterize the limiting distribution of L\'{e}vy-driven
queue with server breakdowns due to the fact that $\{W_t, t\geq0\}$
is not a non-negative process.
\end{remark}

\begin{corollary}\label{cor4.1}
When the system is stable, we have
\begin{align*}
\omega=\mathbb{E}W=\frac{\varphi''(0)}{2\varphi'(0)}+p\frac{\mathbb{E}(W^+)^2-\mathbb{E}(W^-)^2}{2\mathbb{E}\xi_1}
+(1-p)\frac{\mathbb{E}\eta_1^2}{2\mathbb{E}\eta_1},
\end{align*}
\begin{align*}
v=&\mathbb{V}ar
W=\frac{1}{3}\frac{\varphi'''(0)}{\varphi'(0)}-\frac{1}{4}\left(\frac{\varphi''(0)}{\varphi'(0)}\right)^2
+(1-p)\left[\frac{1}{3}\frac{\mathbb{E}\eta_1^3}{\mathbb{E}\eta_1}
-\frac{1}{4}\left(\frac{\mathbb{E}\eta_1^2}{\mathbb{E}\eta_1}\right)^2\right]\\
&+p\left[\frac{1}{3}\frac{\mathbb{E}(W^+)^3-\mathbb{E}(W^-)^3}{\mathbb{E}\xi_1}
-\frac{1}{4}\left(\frac{\mathbb{E}(W^+)^2-\mathbb{E}(W^-)^2}{\mathbb{E}\xi_1}\right)^2\right].
\end{align*}
\end{corollary}

Next, we analyze the system busy period which is defined as the time
period which starts at the epoch when a vacation is completed and
the server begins processing the fluid and ends at the next
 epoch when the system is empty. Note that a busy period includes the
 normal processing time of the fluid in the system and some possible repair times of the server due to
the server breakdowns.
\begin{lemma}\label{lem4.3}
Define
$T_n=\inf\{t|Y_t+\sum_{k=1}^{N_t^R}\xi_k+\sum_{i=1}^nB_i=0\}$,
$n\geq1$, where $\{B_i, i\geq1\}$ is a positive i.i.d. sequence with
$\mathbb{E}B_1<\infty$. If $0\leq p<1$, then
\begin{align}\label{eq4.9}
\mathbb{E}T_n=\frac{n\mathbb{E}B_1}{(1-p)\varphi'(0)}.
\end{align}
\end{lemma}
\begin{proof}
Let $T_n^b=\inf\{t|Y_t+\sum_{k=1}^{N_t^R}\xi_k+nb=0\}$, $b>0$,
$n\geq1$. Since $Y$ is a spectrally positive L\'{e}vy process,
$\{T_n^b, n\geq1\}$ is a random walk with $0<T_n^b<T_{n+1}^b$ w.p.1.
Thus, $\mathbb{E}T_n^b=n\mathbb{E}T_1^b$. Since $\mathbb{E}Y_t<0$
and $0\leq p<1$,
$\mathbb{E}(Y_t+\sum_{k=1}^{N_t^R}\xi_k)=-(1-p)\varphi'(0)t<0$. By
Theorem 8.4.4 of Chung (1974), we have $\mathbb{E}T_1^b<\infty$.
Since $\mathbb{E}|Y_t|<\infty$, $t^{-1}Y_t\to \mathbb{E}Y_1$ w.p.1
as $t\to\infty$ by the Kolmogorov strong law of large numbers.
Taking into account the fact that
$Y_{T_n^b}+\sum_{k=1}^{N_{T_n^b}^R}=-nb$, we get
\begin{align*}
-b&=n^{-1}\left(Y_{T_n^b}+\sum_{k=1}^{N_{T_n^b}^R}\xi_k\right)=n^{-1}T_n^b\left[(T_n^b)^{-1}Y_{T_n^b}+(T_n^b)^{-1}
\sum_{k=1}^{N_{T_n^b}^R}\xi_k\right]\\
&\longrightarrow
\mathbb{E}T_1^b[\mathbb{E}Y_1+(\mathbb{E}N_1^R)(\mathbb{E}\xi_1)]=\mathbb{E}T_1^b[-(1-p)\varphi'(0)]
\ \ \ \mbox{w.p.1 as} \ \ n\to\infty.
\end{align*}
So that $\mathbb{E}T_n^b=\frac{nb}{(1-p)\varphi'(0)}$. Finally, by
conditioning and unconditioning, we obtain \eqref{eq4.9}.
\end{proof}

By the argument of Lemma \ref{lem4.3}, we obtain the following
result.
\begin{theorem}\label{thm4.2}
Let $T=\inf\{t|W_t=0\}$, where $W_0$ is distributed according to the
stationary distribution. Then the mean length of the busy period is
given by
\begin{align*}
\mathbb{E}T=\frac{\mathbb{E}W}{(1-p)\varphi'(0)}.
\end{align*}
\end{theorem}
\begin{remark}
Note that for the $M/G/1$ queue with server breakdowns and multiple
vacations, if $B_i$ is distributed as a service time, then $T_n$ is
the $n$-order busy period which is defined as the time period which
starts when a vacation is completed and there are $n$ customers in
the system and ends at the next departure epoch when the system is
empty.
\end{remark}

\section{The transient distribution}\label{sec5}
In this section, we focuses on analyzing the transient distribution
of $W_t$ in terms of Laplace-Stieltjes transform, for some $t>0$,
conditional on $W_0=x$.
\begin{theorem}\label{thm5.1}
Let $T$ be exponentially distributed with mean $1/\gamma$,
independently of $Y$, $\tau_n$ and $\sigma_n$, $n\geq1$. For
$\theta>0$ and $x\geq0$,
\begin{align}\label{eq5.1}
\mathbb{E}_xe^{-\theta
W_T}=&\frac{\gamma}{\varphi(\theta)-\gamma}e^{-\varphi^{-1}(\gamma)x}\notag\\
&\times \left\{p\frac{\mathbb{E}e^{-\theta W^-}-\mathbb{E}e^{-\theta
W^+}}{\mathbb{E}e^{-\varphi^{-1}(\gamma)
W^-}-\mathbb{E}e^{-\varphi^{-1}(\gamma)
W^+}}+(1-p)\frac{1-\mathbb{E}e^{-\theta\eta_1}}{1-\mathbb{E}e^{-\varphi^{-1}(\gamma)\eta_1}}
-e^{-[\theta-\varphi^{-1}(\gamma)]x}\right\}.
\end{align}
\end{theorem}
\begin{proof}
By Lemma \ref{lem4.1}, we have
\begin{align*}
0&=\mathbb{E}M_T=\varphi(\theta)\int_0^\infty\int_0^t \gamma
e^{-\gamma t} e^{-\theta W_s}dsdt+e^{-\theta
x}-\mathbb{E}_xe^{-\theta W_T}\\
&\ \ \
-\mathbb{E}\sum_{k=1}^{N_T^R}\left[e^{-\theta(W_{\tau_k}-\xi_k)}-e^{-\theta
W_{\tau_k}}\right]-\mathbb{E}\sum_{k=1}^{N_T^V}(1-e^{-\theta\eta_k})\\
&=\frac{\varphi(\theta)}{\gamma}\mathbb{E}_xe^{-\theta
W_T}+e^{-\theta x}-\mathbb{E}_xe^{-\theta W_T}
-\mathbb{E}N_T^R\left[\mathbb{E}e^{-\theta W^-}-\mathbb{E}e^{-\theta
W^+}\right]\\
& \ \ \ -\mathbb{E}N_T^V(1-\mathbb{E}e^{-\theta\eta_1}).
\end{align*}
So that
\begin{align*}
\mathbb{E}_xe^{-\theta W_T}=\frac{\gamma}{\varphi(\theta)-\gamma}
\left\{ \mathbb{E}N_T^R\left[\mathbb{E}e^{-\theta
W^-}-\mathbb{E}e^{-W^+}\right]
+\mathbb{E}N_T^V(1-\mathbb{E}e^{-\theta\eta_1})-e^{-\theta x}
\right\}.
\end{align*}
Note that when $\mathbb{E}Y_1<0$, $\varphi(\theta)$ is increasing on
$[0,\infty)$. Therefore, the inverse of $\varphi(\theta)$ is well
defined on $[0,\infty)$. Hence, the equation
$\varphi(\theta)=\gamma$ has exactly one root
$\theta=\varphi^{-1}(\gamma)$ on $[0,\infty)$. Finally, using the
fact that the root of the denominator should be a root of the
numerator as well, we can obtain \eqref{eq5.1}.
\end{proof}

Next, we derive explicitly the Laplace transform corresponding to
the correlation of the virtual waiting time process:
\begin{align*}
c(t)=\frac{\mathbb{C}ov(W_0, W_t)}{\sqrt{\mathbb{V}ar
W_0\cdot\mathbb{V}ar
W_t}}=\frac{\mathbb{E}(W_0W_t)-(\mathbb{E}W_0)^2}{\mathbb{V}ar
W_0}=\frac{\mathbb{E}(W_0W_t)-\omega^2}{v}.
\end{align*}
Here, we assume the system is in steady-state at time 0.
\begin{theorem}\label{thm5.2}
For $\theta>0$ and $\omega$, $v$ as in Corollary \ref{cor4.1},
\begin{align*}
\int_0^\infty c(t)e^{-\theta
t}dt=&\frac{1}{\theta}-\frac{\omega\varphi'(0)}{v\theta^2}+
\frac{\mathbb{E}(We^{-\varphi^{-1}(\theta)W})}{v\theta}\notag\\
&\times
\left\{\frac{p\mathbb{E}\xi_1}{\mathbb{E}e^{-\varphi^{-1}(\theta)
W^-}-\mathbb{E}e^{-\varphi^{-1}(\theta)
W^+}}+\frac{(1-p)\mathbb{E}\eta_1}{1-\mathbb{E}e^{-\varphi^{-1}(\theta)\eta_1}}\right\},
\end{align*}
where
\begin{align*}
\mathbb{E}(We^{-\varphi^{-1}(\theta)W})=&\left[\frac{\varphi'(0)}{\varphi(\theta)}
-\frac{\theta\varphi'(0)\varphi'(\theta)}{\varphi^2(\theta)}\right]\left[p\frac{\mathbb{E}e^{-\theta
W^-}-\mathbb{E}e^{-\theta
W^+}}{\theta\mathbb{E}\xi_1}+(1-p)\frac{1-\mathbb{E}e^{-\theta\eta_1}}{\theta\mathbb{E}\eta_1}\right]\notag\\
&+\frac{\varphi'(0)}{\varphi(\theta)}\frac{p}{\mathbb{E}\xi_1}\left[\mathbb{E}(W^+e^{-\theta
W^+})-\mathbb{E}(W^-e^{-\theta W^-})-\frac{\mathbb{E}e^{-\theta
W^-}-\mathbb{E}e^{-\theta W^+}}{\theta}\right]\notag\\
&+\frac{\varphi'(0)}{\varphi(\theta)}\frac{1-p}{\mathbb{E}\eta_1}\left[\mathbb{E}(\eta_1e^{-\theta
\eta_1})-\frac{1-\mathbb{E}e^{-\theta \eta_1}}{\theta}\right].
\end{align*}
\end{theorem}
\begin{proof}
Let $T$ be exponentially distributed with mean $1/\theta$. From
\eqref{eq5.1}, we get
\begin{align}\label{eq5.4}
\int_0^\infty \theta e^{-\theta
t}\mathbb{E}_xW_tdt=&-\frac{\varphi'(0)}{\theta}+
e^{-\varphi^{-1}(\theta)x}
\left[\frac{p\mathbb{E}\xi_1}{\mathbb{E}e^{-\varphi^{-1}(\theta)
W^-}-\mathbb{E}e^{-\varphi^{-1}(\theta)
W^+}}+\frac{(1-p)\mathbb{E}\eta_1}{1-\mathbb{E}e^{-\varphi^{-1}(\theta)\eta_1}}\right]+x.
\end{align}
Straightforward calculus yields
\begin{align}\label{eq5.5}
\int_0^\infty c(t)e^{-\theta t}dt&=\frac{1}{v}\int_0^\infty
[\mathbb{E}(W_0W_t)-\omega^2]e^{-\theta t}dt\notag\\
&=\frac{1}{v}\int_0^\infty\int_0^\infty x\mathbb{E}_xW_te^{-\theta
t}d\mathbb{P}(W_0\leq x)dt-\frac{\omega^2}{v\theta}.
\end{align}
Substituting \eqref{eq5.4} into \eqref{eq5.5}, we obtain
\begin{align*}
&\int_0^\infty c(t)e^{-\theta t}dt=-\frac{\omega^2}{v\theta}+\notag\\
&\int_0^\infty \frac{x}{v\theta}\left\{-\frac{\varphi'(0)}{\theta}+
e^{-\varphi^{-1}(\theta)x}
\left[\frac{p\mathbb{E}\xi_1}{\mathbb{E}e^{-\varphi^{-1}(\theta)
W^-}-\mathbb{E}e^{-\varphi^{-1}(\theta)
W^+}}+\frac{(1-p)\mathbb{E}\eta_1}{1-\mathbb{E}e^{-\varphi^{-1}(\theta)\eta_1}}\right]+x\right\}d\mathbb{P}(W_0\leq
x)\notag\\
&=\frac{1}{\theta}-\frac{\omega\varphi'(0)}{v\theta^2}+
\frac{\mathbb{E}(We^{-\varphi^{-1}(\theta)W})}{v\theta}
\left\{\frac{p\mathbb{E}\xi_1}{\mathbb{E}e^{-\varphi^{-1}(\theta)
W^-}-\mathbb{E}e^{-\varphi^{-1}(\theta)
W^+}}+\frac{(1-p)\mathbb{E}\eta_1}{1-\mathbb{E}e^{-\varphi^{-1}(\theta)\eta_1}}\right\},
\end{align*}
where $\mathbb{E}(We^{-\varphi^{-1}(\theta)W})$ can be obtained by
differentiating \eqref{eq4.4}. This completes the proof.
\end{proof}

\section{Stochastic decompositions}\label{sec6}
In this section, we identify some stochastic decomposition
properties of the virtual waiting time for our fluid queueing
system. Stochastic decomposition properties have been studied in
many vacation models. The classical stochastic decomposition
property shows that the steady-state system size at an arbitrary
point can be represented as the sum of two independent random
variable, one of which is the system size of the corresponding
standard queueing system without server vacations and the other
random variable depends on the meaning of vacations in specific
cases (see Fuhrmann and Cooper \cite{Fuhrmann1985} and Doshi
\cite{Doshi1990}). In addition, stochastic decomposition properties
have also been held for L\'{e}vy-driven queues with interruptions as
well as c\`{a}dl\`{a}g processes with certain secondary jump inputs
(see Kella and Whitt \cite{Kella1991} and Ivanovs and Kella
\cite{Ivanovs2013}). It's worth mentioning that, recently, Boxma and
Kella \cite{Boxma2014} generalize known workload decomposition
results for L\'{e}vy queues with secondary jump inputs and queues
with server vacations or service interruptions. In particular, in
the context of our fluid queueing system, we have the following
decomposition results.
\begin{lemma}\label{lem6.1}
There are two random variables $U$ and $V$ such that
\begin{align}\label{eq6.1}
\frac{\mathbb{E}e^{-\theta W^-}-\mathbb{E}e^{-\theta
W^+}}{\theta\mathbb{E}\xi_1}=\mathbb{E}e^{-\theta U},
\end{align}
\begin{align}\label{eq6.2}
\frac{1-\mathbb{E}e^{-\theta\eta_1}}{\theta\mathbb{E}\eta_1}=\mathbb{E}e^{-\theta
V}.
\end{align}
\end{lemma}
\begin{proof}
The left hand side of \eqref{eq6.1} is the Laplace transformation of
the function
\begin{align*}
f(x)=\frac{\mathbb{P}(W^+>x)-\mathbb{P}(W^->x)}{\mathbb{E}\xi_1},
\end{align*}
which is a bona fide probability density function.

The left hand side of \eqref{eq6.2} is the Laplace transformation of
the function
\begin{align*}
g(x)=\frac{\mathbb{P}(\eta_1>x)}{\mathbb{E}\eta_1},
\end{align*}
which is a bona fide probability density function (stationary
residual life density of $\eta_1$).
\end{proof}

Applying Lemma \ref{lem6.1}, we can get the following theorem to
characterize the stochastic decomposition property.
\begin{theorem}\label{thm6.1}
Under the conditions of Theorem \ref{thm4.1}, the distribution of
$W$ is the convolution of two distributions one of which is the
distribution of $R$ in Section 2 and the other is the mixture of the
distribution of $U$ with probability $p$ and the stationary residual
life distribution of $\eta_1$ with probability $1-p$.
\end{theorem}

Furthermore, under some conditions, we have another stochastic
decomposition result.
\begin{lemma}\label{lem6.2}
Under the conditions of Theorem \ref{thm4.1}, if positive random
variables $\xi_n$, $n\geq1$ are independent of $W_0$, $Y_{\tau_n}$
and $\sum_{i=1}^{N_{\tau_n}^V}\eta_i$, then we have
\begin{align}\label{eq6.3}
\mathbb{E}e^{-\theta
W}=\frac{\theta\varphi'(0)}{\varphi(\theta}\left[p\frac{1-\mathbb{E}e^{-\theta\xi_1}}{\theta\mathbb{E}\xi_1}
\mathbb{E}e^{-\theta W^-}
+(1-p)\frac{1-\mathbb{E}e^{-\theta\eta_1}}{\theta\mathbb{E}\eta_1}\right].
\end{align}
\end{lemma}

\begin{proof}
Since $\xi_k$ is independent of
$W_{\tau_k}-\xi_k=W_0+Y_{\tau_n}+\sum_{i=1}^{k-1}\xi_i+\sum_{i=1}^{N_{\tau_k}^V}\eta_i$,
\begin{align}\label{eq6.4}
\frac{\mathbb{E}e^{-\theta (W_{\tau_1}-\xi_1)}-\mathbb{E}e^{-\theta
W_{\tau_1}}}{\theta\mathbb{E}\xi_1}&=\frac{\mathbb{E}e^{-\theta
(W_{\tau_1}-\xi_1)}(1-e^{-\theta\xi_1})}{\theta\mathbb{E}\xi_1}\notag\\
&=\mathbb{E}e^{-\theta
(W_{\tau_1}-\xi_1)}\frac{1-\mathbb{E}e^{-\theta\xi_1}}{\theta\mathbb{E}\xi_1}\notag\\
&=\mathbb{E}e^{-\theta
W^-}\frac{1-\mathbb{E}e^{-\theta\xi_1}}{\theta\mathbb{E}\xi_1}.
\end{align}

Substituting \eqref{eq6.4} into \eqref{eq4.4} yields \eqref{eq6.3}.

\end{proof}

From Lemma \ref{lem6.2}, we obtain the following theorem which
provides another stochastic decomposition.
\begin{theorem}
If the assumptions of Lemma 6.2 are satisfied, then the distribution
of $U$ is the convolution of the distribution of $W^-$ and the
stationary residual life distribution of $\xi_1$.
\end{theorem}


\bibliographystyle{apt}
\bibliography{levy}

\end{document}